\documentclass[12pt,leqno]{article}
\usepackage{hyperref}
\usepackage{soul}
\usepackage[doc]{optional}
\usepackage{srcltx}
\usepackage{xcolor}
\usepackage{exscale,relsize}
\usepackage{amsmath}
\usepackage{amsfonts}
\usepackage{amssymb}
\usepackage{calc}
\usepackage{theorem}
\usepackage{pifont}      %needed by dingautolist
\usepackage{graphicx}

\newcommand{\scal}[2]{\langle{{#1},{#2}}\rangle}

\newcommand{\RR}{\ensuremath{\mathbb R}}

\newcommand{\RX}{\ensuremath{\,\left]-\infty,+\infty\right]}}
\newcommand{\RXX}{\ensuremath{\,\left[-\infty,+\infty\right]}}

\newcommand{\NN}{\ensuremath{\mathbb N}}

\newcommand{\thalb}{\ensuremath{\tfrac{1}{2}}}

\newcommand{\To}{\ensuremath{\rightrightarrows}}

\newcommand{\dom}{\ensuremath{\operatorname{dom}}}

\newcommand{\gra}{\ensuremath{\operatorname{gra}}}

\newcommand{\intdom}{\ensuremath{\operatorname{int}\operatorname{dom}}\,}
\newcommand{\inte}{\ensuremath{\operatorname{int}}}

\newcommand{\ran}{\ensuremath{\operatorname{ran}}}

\newcommand{\Id}{\ensuremath{\operatorname{Id}}}

\newcommand{\menge}[2]{\big\{{#1} \mid {#2}\big\}}

\newcommand{\Elag}{\ensuremath{\varepsilon}}
\newcommand{\ve}{\ensuremath{\varepsilon}}

\renewcommand{\phi}{\ensuremath{\varphi}}

\newtheorem{theorem}{Theorem}[section]
\newtheorem{lemma}[theorem]{Lemma}
\newtheorem{fact}[theorem]{Fact}
\newtheorem{corollary}[theorem]{Corollary}
\newtheorem{proposition}[theorem]{Proposition}
\newtheorem{definition}[theorem]{Definition}

\theoremstyle{plain}{\theorembodyfont{\rmfamily}
}
\theoremstyle{plain}{\theorembodyfont{\rmfamily}
}
\theoremstyle{plain}{\theorembodyfont{\rmfamily}
}
\theoremstyle{plain}{\theorembodyfont{\rmfamily}
\newtheorem{example}[theorem]{Example}}
\theoremstyle{plain}{\theorembodyfont{\rmfamily}
\newtheorem{remark}[theorem]{Remark}}

\theoremstyle{plain}{\theorembodyfont{\rmfamily}
}

\def\endproof{\ensuremath{\quad \hfill \blacksquare}}

\begin{document}

%\sffamily

\title{\sffamily{Monotone Operators without Enlargements
 }}

\author{
Jonathan M.
Borwein\thanks{CARMA, University of Newcastle, Newcastle, New South
Wales 2308, Australia. E-mail:
\texttt{jonathan.borwein@newcastle.edu.au}.  Distinguished Professor
King Abdulaziz University, Jeddah. },\;
Regina Burachik\thanks{School of Mathematics and Statistics,
 University of South Australia, Mawson Lakes, SA 5095, Australia. E-mail:
\texttt{regina.burachik@unisa.edu.au}.},\;
and Liangjin\
Yao\thanks{Mathematics, Irving K.\ Barber School, University of British Columbia,
Kelowna, B.C. V1V 1V7, Canada.
E-mail:  \texttt{ljinyao@interchange.ubc.ca}.}}

\date{October 12, 2011}
\maketitle

\begin{abstract} \noindent
Enlargements have proven to be useful tools for studying maximally
monotone mappings. It is therefore natural to ask in which cases the
enlargement does not change the original mapping.  Svaiter has
recently characterized non-enlargeable operators in reflexive Banach
spaces and has also given some partial results in the nonreflexive case.
In the present paper, we provide another characterization of
non-enlargeable operators in nonreflexive Banach spaces under a
closedness assumption on the graph.  Furthermore, and still for
general Banach spaces, we present a new proof of the maximality of the
sum of two maximally monotone linear relations. We also present a new proof of the
maximality of the sum of a maximally monotone linear relation and a normal cone operator
when the domain of the linear relation intersects the interior of the
domain of the normal cone.
\end{abstract}

\noindent


\noindent

\section{Introduction}

Maximally monotone operators have proven to be a significant class of
objects in both modern Optimization and Functional Analysis. They
extend both the concept of subdifferentials of convex functions, as
well as that of a positive semi-definite function. Their study in the context of
Banach spaces, and in particular nonreflexive ones, arises naturally in
the theory of partial differential equations, equilibrium problems,
and variational inequalities. For a detailed study of these operators,
see, e.g., \cite{Bor1,Bor2,Bor3}, or the books
\cite{BC2011,BorVan,BurIus,ph,Si,Si2,RockWets,Zalinescu,Zeidler}.

A useful tool for studying or proving properties of a maximally
monotone operator $A$ is the concept of the ``enlargement of $A$''. A
main example of this usefulness is Rockafellar's proof of maximality
of the subdifferential of a convex function (Fact \ref{SubMR} below), which uses
the concept of $\ve$-subdifferential. The latter is an enlargement of
the subdifferential introduced in \cite{BR}.

Broadly speaking, an enlargement is a multifunction which approximates
the original maximally monotone operator in a convenient way.
Another useful way to study a maximally monotone operator is by associating to
it a convex function called the Fitzpatrick function.  The latter was
introduced by Fitzpatrick in \cite{Fitz88} and its connection with
enlargements, as shown in \cite{BurSva:1}, is contained in
\eqref{Enl:1} below.

Our first aim in the present paper is to provide further
characterizations of maximally monotone operators which are not
enlargeable, in the setting of possibly nonreflexive Banach spaces
(see Section~\ref{secneo}). In other words, in which cases the
enlargement does not change the graph of a maximally monotone mapping
defined in a Banach space? We address this issue Corollary \ref{CEl:2},
under a closedness assumption on the graph of the operator.

Our other aim is to use the Fitzpatrick function to derive new
results which establish the maximality of the sum of two maximally monotone
operators in nonreflexive spaces (see Section~\ref{secsumo}). First, we provide a different proof
of the maximality of the sum of two  maximally monotone
linear relations. Second, we provide a proof of the maximality of the sum of
a  maximally monotone linear relation and a normal cone operator when
the domain of the operator intersects the interior of the domain of
the normal cone.

\section{Technical Preliminaries}

Throughout this paper, $X$ is a real Banach space with norm
$\|\cdot\|$, and $X^*$ is the continuous dual of $X$. The spaces $X$
and $X^*$ are paired by the duality pairing, denoted as
$\scal{\cdot}{\cdot}$. The space $X$ is identified with its canonical
image in the bidual space $X^{**}$.  Furthermore, $X\times X^*$ and
$(X\times X^*)^*: = X^*\times X^{**}$ are  paired via
$\scal{(x,x^*)}{(y^*,y^{**})}:= \scal{x}{y^*} + \scal{x^*}{y^{**}}$,
where $(x,x^*)\in X\times X^*$ and $(y^*,y^{**}) \in X^*\times
X^{**}$.

Let $A\colon X\To X^*$ be a \emph{set-valued operator} (also known as
a multifunction) from $X$ to $X^*$, i.e., for every $x\in X$,
$Ax\subseteq X^*$, and let $\gra A:= \menge{(x,x^*)\in X\times
  X^*}{x^*\in Ax}$ be the \emph{graph} of $A$.  The \emph{domain} of
$A$ is $\dom A:= \menge{x\in X}{Ax\neq\varnothing}$, and $\ran
A:=A(X)$ for the \emph{range} of $A$. Recall that $A$ is
\emph{monotone} if
\begin{equation}
\scal{x-y}{x^*-y^*}\geq 0,\quad \forall (x,x^*)\in \gra A\;
\forall (y,y^*)\in\gra A,
\end{equation}
and \emph{maximally monotone} if $A$ is monotone and $A$ has
 no proper monotone extension
(in the sense of graph inclusion).
Let $A:X\rightrightarrows X^*$ be monotone and $(x,x^*)\in X\times X^*$.
 We say $(x,x^*)$ is \emph{monotonically related to}
$\gra A$ if
\begin{align*}
\langle x-y,x^*-y^*\rangle\geq0,\quad \forall (y,y^*)\in\gra
A.\end{align*}
Let $A:X\rightrightarrows X^*$ be maximally monotone. We say $A$ is
\emph{of type (FPV)} if  for every open convex set $U\subseteq X$ such that
$U\cap \dom A\neq\varnothing$, the implication
\begin{equation*}
x\in U\text{and}\,(x,x^*)\,\text{is monotonically related to $\gra A\cap U\times X^*$}
\Rightarrow (x,x^*)\in\gra A
\end{equation*}
holds. Maximally monotone operators of type (FPV) are relevant primarily in
the context of nonreflexive Banach spaces. Indeed, it follows from
\cite[Theorem 44.1]{Si2} and a well-known result from \cite{Rock701}
that every maximally monotone operator defined in a reflexive Banach
space is of type (FPV). As mentioned in  \cite[\S44]{Si2}, an example
of a maximally monotone operator which is not of type (FPV) has not been found yet.

Let $A:X\rightrightarrows X^*$ be monotone such that $\gra A\neq\varnothing$. The
\emph{Fitzpatrick function} associated with $A$ is defined by
\begin{equation*}
F_A\colon X\times X^*\to\RX\colon
(x,x^*)\mapsto \sup_{(a,a^*)\in\gra A}
\big(\scal{x}{a^*}+\scal{a}{x^*}-\scal{a}{a^*}\big).
\end{equation*}
When $A$ is maximally monotone, a fundamental property of the
Fitzpatrick function $F_A$ (see Fact \ref{FFc}) is that
\begin{align}
F_A(x,x^*)&\ge \scal{x}{x^*} \hbox{ for all }(x,x^*)\in X\times X^*,\label{FFa}\\
 F_A(x,x^*)&= \scal{x}{x^*} \hbox{ for all }(x,x^*)\in \gra A.\label{FFb}
\end{align}
Hence, for a fixed $\varepsilon\geq0$, the set of pairs $(x,x^*)$ for
which $F_A(x,x^*)\le \langle x,x^*\rangle+\varepsilon$ contains the
graph of $A$. This motivates the definition of enlargement of $A$ for
a general monotone mapping $A$, which is as follows.

Let $\varepsilon\geq0$. We define $A_{\Elag}:X\rightrightarrows X^*$ by
\begin{align}\gra A_{\Elag}&:=\Big\{(x,x^*)\in X\times X^*\mid\langle x^*-y^*,x-y\rangle
\geq-\varepsilon,\; \forall (y, y^*)\in \gra A\Big\}\nonumber\\
&=\Big\{(x,x^*)\in X\times X^*\mid F_A (x,x^*)\leq \langle x,x^*\rangle+\varepsilon\Big\}.\label{Enl:1}
\end{align}

Let $A:X\rightrightarrows X^*$ be monotone.  We say $A$ is
\emph{enlargeable} if $\gra A\varsubsetneqq \gra A_{\Elag}$ for some
$\varepsilon\geq0$, and $A$ is \emph{non-enlargeable} if $\gra A=\gra
A_{\Elag}$ for every $\varepsilon\geq0$.  Lemma 23.1 in \cite{Si2}
proves that if a proper and convex function verifies \eqref{FFa}, then
the set of all pairs $(x,x^*)$ at which \eqref{FFb} holds is a
monotone set. Therefore, if $A$ is non-enlargeable then it must be
maximally monotone.

We adopt the notation used in the books \cite[Chapter 2]{BorVan} and
\cite{Bor1, Si, Si2}. Given a subset $C$ of $X$, $\inte C$ is the
\emph{interior} of $C$, $\overline{C}$ is the \emph{norm closure} of
$C$.  The \emph{support function} of $C$, written as $\sigma_C$, is
defined by $\sigma_C(x^*):=\sup_{c\in C}\langle c,x^*\rangle$.  The
\emph{indicator function} of $C$, written as $\iota_C$, is defined at
$x\in X$ by
\begin{align}
\iota_C (x):=\begin{cases}0,\,&\text{if $x\in C$;}\\
+\infty,\,&\text{otherwise}.\end{cases}\end{align}
For every $x\in X$, the \emph{normal cone operator} of $C$ at $x$
is defined by $N_C(x):= \menge{x^*\in
X^*}{\sup_{c\in C}\scal{c-x}{x^*}\leq 0}$, if $x\in C$; and $N_C(x):=\varnothing$,
if $x\notin C$.
 The \emph{closed unit
ball} is $B_X:=\menge{x\in X}{\|x\|\leq 1}$, and
$\NN:=\{1,2,3,\ldots\}$.

If $Z$ is a real  Banach space with dual $Z^*$ and a set $S\subseteq
Z$, we denote $S^\bot$ by $S^\bot := \{z^*\in Z^*\mid\langle
z^*,s\rangle= 0,\quad \forall s\in S\}$.
The \emph{adjoint} of an operator  $A$, written $A^*$, is defined by
\begin{equation*}
\gra A^* :=
\big\{(x^{**},x^*)\in X^{**}\times X^*\mid(x^*,-x^{**})\in(\gra A)^{\bot}\big\}.
\end{equation*}
We will be interested in monotone operators which are \emph{linear
  relations}, i.e., such that $\gra A$ is a linear subspace. Note that in this situation,
$A^*$ is also a linear relation. Moreover, $A$ is \emph{symmetric} if $\gra A
\subseteq\gra A^*$. Equivalently, for all $(x,x^*), (y,y^*)\in\gra A$ it holds that
\begin{equation}\label{sym}
\scal{x}{y^*}=\scal{y}{x^*}.
\end{equation}
We say that a linear relation $A$ is
\emph{skew}  if $\gra A \subseteq \gra (-A^*)$. Equivalently, for all $(x,x^*)\in\gra A$ we have
\begin{equation}\label{skew}
\langle x,x^*\rangle=0.
\end{equation}

We define the \emph{symmetric part} a of $A$ via
\begin{equation}
\label{Fee:1}
A_+ := \thalb A + \thalb A^*.
\end{equation}
It is easy to check that $A_+$ is symmetric.

Let $f\colon X\to \RX$. Then $\dom f:= f^{-1}(\RR)$ is the
\emph{domain} of $f$, and $f^*\colon X^*\to\RXX\colon x^*\mapsto
\sup_{x\in X}(\scal{x}{x^*}-f(x))$ is the \emph{Fenchel conjugate} of
$f$. We denote by $\overline{f}$ the lower semicontinuous hull of $f$.
We say that $f$ is proper if $\dom f\neq\varnothing$.  Let $f$ be
proper. The \emph{subdifferential} of $f$ is defined by
   $$\partial f\colon X\To X^*\colon
   x\mapsto \{x^*\in X^*\mid(\forall y\in
X)\; \scal{y-x}{x^*} + f(x)\leq f(y)\}.$$
For $\varepsilon \geq 0$,
the \emph{$\varepsilon$--subdifferential} of $f$ is defined by
  \begin{align*}\partial_{\varepsilon} f\colon X\To X^*\colon
   x\mapsto \menge{x^*\in X^*}{(\forall y\in
X)\; \scal{y-x}{x^*} + f(x)\leq f(y)+\varepsilon}.
\end{align*}
Note that $\partial f = \partial_{0}f$.

Relatedly, we say $A$ is
of  Br{\o}nsted-Rockafellar (BR) type   \cite{Si2,BorVan}
if whenever $(x,x^*)\in X\times X^*$, $\alpha,\beta>0$ while
\begin{align*}\inf_{(a,a^*)\in\gra A} \langle x-a,x^*-a^*\rangle
>-\alpha\beta\end{align*} then there exists $(b,b^*)\in\gra A$ such
that $\|x-b\|<\alpha,\|x^*-b^*\|<\beta$.
The name is motivated by the celebrated theorem of  Br{\o}nsted and Rockafellar \cite{Si2,BorVan}
which can be stated now as saying that all closed convex subgradients are of type (BR).

Let $g\colon X\to\RX$. The \emph{inf-convolution} of $f$ and $g$, $f\Box g$, is defined by
\begin{align*}f\Box g: x\rightarrow\inf_{y\in X} \left[f(y)+g(x-y)\right].
\end{align*}
Let $Y$ be another real Banach space. We set  $P_X: X\times Y\rightarrow
X\colon (x,y)\mapsto x$. We denote $\Id: X\rightarrow X$ by the \emph{identity mapping}.

Let $F_1, F_2\colon X\times Y\rightarrow\RX$.
Then the \emph{partial inf-convolution} $F_1\Box_2 F_2$
is the function defined on $X\times Y$ by
\begin{equation}\label{infconv}
F_1\Box_2 F_2\colon (x,y)\mapsto \inf_{v\in Y}\left[
  F_1(x,y-v)+F_2(x,v)\right].
\end{equation}

 \section{Auxiliary results}\label{s:aux}

We collect in this section some facts we will use later on. These
facts involve convex functions, maximally monotone operators and Fitzpatrick functions.

\begin{fact}\emph{(See \cite[Proposition~3.3 and Proposition~1.11]{ph}.)}
\label{pheps:1}Let $f:X\rightarrow\RX$ be a lower semicontinuous convex
 and $\intdom f\neq\varnothing$.
Then $f$ is continuous on $\intdom f$
 and $\partial f(x)\neq\varnothing$ for every $x\in\intdom f$.
\end{fact}

 \begin{fact}[Rockafellar] \label{f:F4}
\emph{(See {\cite[Theorem~3(a)]{Rock66}},
{\cite[Corollary~10.3 and Theorem~18.1]{Si2}}, or
{\cite[Theorem~2.8.7(iii)]{Zalinescu}}.)}
Let $f,g: X\rightarrow\RX$ be proper convex functions.
Assume that there exists a point $x_0\in\dom f \cap \dom g$
such that $g$ is continuous at $x_0$.
Then for every $z^*\in X^*$,
there exists $y^*\in X^*$ such that
\begin{equation}
(f+g)^*(z^*) = f^*(y^*)+g^*(z^*-y^*).
\end{equation}
\end{fact}

\begin{fact}[Rockafellar]\label{SubMR}\emph{(See \cite[Theorem~A]{Rock702},
 \cite[Theorem~3.2.8]{Zalinescu}, \cite[Theorem~18.7]{Si2} or \cite[Theorem~2.1]{MSV})}
 Let $f:X\rightarrow\RX$ be a proper lower semicontinuous convex function.
Then $\partial f$ is maximally monotone.
\end{fact}

\begin{fact}[Attouch-Br\'ezis]\emph{(See \cite[Theorem~1.1]{AtBrezis}
 or \cite[Remark~15.2]{Si2}).}\label{AttBre:1}
Let $f,g: X\rightarrow\RX$ be proper lower semicontinuous  and convex.
Assume that $
\bigcup_{\lambda>0} \lambda\left[\dom f-\dom g\right]$
is a closed subspace of $X$.
Then
\begin{equation*}
(f+g)^*(z^*) =\min_{y^*\in X^*} \left[f^*(y^*)+g^*(z^*-y^*)\right],\quad \forall z^*\in X^*.
\end{equation*}
\end{fact}

Fact \ref{SubMR} above relates a convex function with maximal
monotonicity. Fitzpatrick functions go in the opposite way: from
maximally monotone operators to convex functions.

\begin{fact}[Fitzpatrick]\label{FFc}
\emph{(See {\cite[Corollary~3.9]{Fitz88}} and
\cite{Bor1,BorVan}.)} \label{f:Fitz} Let $A\colon X\To X^*$ be
maximally monotone.
Then for every $(x,x^*)\in X\times X^*$, the inequality $\scal{x}{x^*}\leq F_A(x,x^*)$
is true,
and the equality holds if and only if $(x,x^*)\in\gra A$.
\end{fact}

It was pointed out in \cite[Problem 31.3]{Si2} that it is unknown
whether $\overline{\dom A}$ is necessarily convex when $A$ is
maximally monotone and $X$ is not reflexive. When $A$ is of type (FPV), the
question was answered positively by using $F_A$.

\begin{fact}[Simons]
\emph{(See \cite[Theorem~44.2]{Si2}.)}
\label{f:referee02c}
Let $A:X\To X^*$ be  of type (FPV). Then
$\overline{\dom A}=\overline{P_X\left[\dom F_A\right]}$ and $\overline{\dom A}$ is convex.
\end{fact}

We observe that when $A$ is of type (FPV) then also $\dom A_{\ve}$ has convex closure.

\begin{remark}
Let  $A$ be of type (FPV)  and fix $\varepsilon\geq0$.
Then by \eqref{Enl:1}, Fact~\ref{f:Fitz} and Fact \ref{f:referee02c}, we have
$\dom A\subseteq\dom A_{\Elag}\subseteq P_X\left[\dom F_A\right]
\subseteq\overline{\dom A}$. Thus we obtain
\[
\overline{\dom A}= \overline{\left[\dom A_{\Elag}\right]}=\overline{P_X\left[\dom F_A\right]},
\]
and this set is convex because $\dom F_A$ is convex. As a result, for
every $A$ of type (FPV) it holds that $\overline{\dom
  A}=\overline{\left[\dom A_{\Elag}\right]}$ and this set is convex.
\end{remark}

We recall below some necessary conditions for a maximally monotone operator to be of type (FPV).

\begin{fact}[Simons]
\emph{(See \cite[Theorem~46.1]{Si2}.)}
\label{f:referee01}
Let $A:X\To X^*$ be a maximally monotone linear relation.
Then $A$ is of type (FPV).
\end{fact}

\begin{fact}[Fitzpatrick-Phelps and Verona-Verona]
\emph{(See \cite[Corollary~3.4]{FitzPh}, \cite[Theorem~3]{VV1} or \cite[Theorem~48.4(d)]{Si2}.)}
\label{f:referee0d}\index{subdifferential operator}
Let $f:X\rightarrow\RX$ be proper, lower semicontinuous, and convex.\index{type (FPV)}
Then $\partial f$ is of type (FPV).
\end{fact}

\begin{fact}\emph{(See \cite[Corollary~3.3]{Yao2}.)}\label{domain:L1}
Let $A:X\To X^*$ be  a maximally monotone linear relation, and $f:X\rightarrow\RX$ be a
proper lower semicontinuous convex function with
 $\dom  A\cap\inte\dom\partial f\neq\varnothing$.
Then $A+\partial f$ is of type  $(FPV)$.
\end{fact}

\begin{fact}[Phelps-Simons]\emph{(See \cite[Corollary 2.6 and Proposition~3.2(h)]{PheSim}.)}\label{F:1}
Let $A\colon X\rightarrow X^*$ be  monotone and linear. Then $A$ is
maximally monotone and continuous.
\end{fact}

\begin{fact}\emph{(See \cite[Theorem~4.2]{BWY3} or \cite[Lemma~1.5]{MarSva2}.)}\label{affine:L1}
Let $A:X\To X^*$ be  maximally monotone such that $\gra A$ is convex.
Then $\gra A$ is affine.
\end{fact}

\begin{fact}[Simons]
\emph{(See \cite[Lemma~19.7 and Section~22]{Si2}.)}
\label{f:referee}
Let $A:X\rightrightarrows X^*$ be a monotone operator such
 that $\operatorname{gra} A$ is convex with $\operatorname{gra} A
\neq\varnothing$.
Then the function
\begin{equation}
g\colon X\times X^* \rightarrow \left]-\infty,+\infty\right]\colon
(x,x^*)\mapsto \langle x, x^*\rangle + \iota_{\operatorname{gra} A}(x,x^*)
\end{equation}
is proper and convex.
\end{fact}

\begin{fact}
\emph{(See \cite[Theorem~3.4 and Corollary~5.6]{Voi1}, or \cite[Theorem~24.1(b)]{Si2}.)}
\label{f:referee1}
Let $A, B:X\To X^*$ be maximally monotone operators. Assume that
$\bigcup_{\lambda>0} \lambda\left[P_X(\dom F_A)-P_X(\dom F_B)\right]$
is a closed subspace.
If
\begin{equation}
F_{A+B}\geq\langle \cdot,\,\cdot\rangle\;\text{on \; $X\times X^*$},
\end{equation}
then $A+B$ is maximally monotone.
\end{fact}

\begin{definition}[Fitzpatrick family]
Let  $A\colon X \To X^*$ be maximally monotone.
The associated \emph{Fitzpatrick family}
$\mathcal{F}_A$ consists of all functions $F\colon
X\times X^*\to\RX$ that are lower semicontinuous and convex,
and that satisfy
$F\geq \scal{\cdot}{\cdot} $, and $F=\scal{\cdot}{\cdot}$ on $\gra A$.
\end{definition}

\begin{fact}[Fitzpatrick]\emph{(See \cite[Theorem~3.10]{Fitz88}
 or \cite{BurSva:1}.)}\label{GRF:2}
Let  $A\colon X \To X^*$ be maximally monotone.
Then for every $(x,x^*)\in X\times X^*$,
\begin{equation*}
F_A(x,x^*) = \min\menge{F(x,x^*)}{F\in \mathcal{F}_A}.
\end{equation*}
\end{fact}

\begin{corollary}\label{GRF:3}
Let  $A\colon X \To X^*$ be a maximally monotone operator such
 that $\operatorname{gra} A$ is convex.
Then for every $(x,x^*)\in X\times X^*$,
\begin{equation*}
F_A(x,x^*) = \min\menge{F(x,x^*)}{F\in \mathcal{F}_A}\quad\text{and}\quad
g(x,x^*)= \max\menge{F(x,x^*)}{F\in \mathcal{F}_A},
\end{equation*}
where $g:= \langle \cdot, \cdot\rangle + \iota_{\operatorname{gra} A}$.
\end{corollary}
\begin{proof}
Apply Fact~\ref{f:referee} and Fact~\ref{GRF:2}.
\end{proof}

\begin{fact}
\emph{(See \cite[Lemma~23.9]{Si2}, or \cite[Proposition~4.2]{BM}.)}
\label{f:referee03}
Let $A, B\colon X\To X^*$ be   monotone operators and $\dom A\cap\dom B\neq\varnothing$.
Then $F_{A+B}\leq F_A\Box_2 F_B$.
\end{fact}

Let $X,Y$ be two real Banach spaces and let $h:X\times Y
\rightarrow\RX$ be a convex function. We say that $h$ is {\em
  separable} if there exist convex functions $h_1:X \rightarrow\RX$
and $h_2:Y \rightarrow\RX$ such that $h(x,y)= h_1(x)+h_2(y)$. This
situation is denoted as $h=h_1\oplus h_2$. We recall below some cases
in which the Fitzpatrick function is separable.

\begin{fact}
\emph{(See \cite[Corollary~5.9]{BBBRW} or \cite[Fact~4.1]{BBWY4}.)}
\label{f:referee04}
Let $C$ be a nonempty closed convex subset of $X$.
Then $F_{N_C}=\iota_C\oplus \iota^*_C$.
\end{fact}

\begin{fact}
\emph{(See \cite[Theorem~5.3]{BBBRW}.)}
\label{f:sub05}
Let  $f:X\rightarrow\RX$
 be a proper lower semicontinuous sublinear function.
Then $F_{\partial f}=f\oplus f^*$ and $\mathcal{F}_A=\big\{
f\oplus f^*\}$.
\end{fact}

\begin{remark}\label{r:sub05}
Let $f$ be as in Fact~\ref{f:sub05}, then
\begin{align}\gra (\partial f)_{\Elag}&=\big\{(x,x^*)\in X\times X^*\mid
f(x)+f^*(x^*)\leq \langle x,x^*\rangle+\varepsilon\big\}\nonumber\\
&=\gra \partial_{\varepsilon}f,\quad \forall \varepsilon\geq0.
\label{Enl:Sub1}
\end{align}
\end{remark}

\begin{fact}[Svaiter]\emph{(See \cite[page~312]{SV}.)}\label{CElL:1}
Let  $A\colon X \To X^*$ be maximally monotone. Then
 $A$ is non-enlargeable if and only if $\gra A=\dom F_A$ and then $\gra A$ is convex.
\end{fact}

It is immediate from the definitions that:

\begin{fact}\label{nEBR} Every non-enlargeable maximally monotone operator is of type (BR).
\end{fact}

Fact \ref{f:sub05} and the subsequent remark refers to a case in which
all enlargements of $A$ coincide, or, equivalently, the Fitzpatrick
family is a singleton. It is natural to deduce that a non-enlargeable
operator will also have a single element in its Fitzpatrick family.

\begin{corollary}\label{CEl:1}
Let  $A\colon X \To X^*$ be maximally monotone. Then
 $A$ is non-enlargeable
if and only if  $F_A=\iota_{\gra A}+\langle\cdot,\cdot\rangle$
and hence  $\mathcal{F}_A=\big\{
\iota_{\gra A}+\langle\cdot,\cdot\rangle\big\}$.
\end{corollary}
\begin{proof}
``$\Rightarrow$":
  By Fact~\ref{CElL:1}, we have $\gra A$ is convex.
  By Fact~\ref{f:Fitz} and Fact~\ref{CElL:1},
we have $F_A=\iota_{\gra A}+\langle \cdot,\cdot\rangle$.
Then by Corollary~\ref{GRF:3}, $\mathcal{F}_A=\big\{
\iota_{\gra A}+\langle\cdot,\cdot\rangle\big\}$.
``$\Leftarrow$":
Apply directly Fact~\ref{CElL:1}.
\end{proof}

\begin{remark}The condition that $\mathcal{F}_A$ is singleton does not guarantee
that $\gra A$ is convex. For example, let $f:X\rightarrow\RX$
 be a proper lower semicontinuous sublinear function.
Then by Fact~\ref{f:sub05}, $\mathcal{F}_A$ is singleton but $\gra \partial f$ is not necessarily convex.
%Hence we cannot weaken the assumptions of Corollary~\ref{CEl:1}.
\end{remark}

\section{Non-Enlargeable  Monotone Linear Relations}
\label{secneo}

We begin with a basic characterization:

\begin{theorem}\label{CElC:01}
 Let $A\colon X \To X^*$ be a maximally monotone linear relation such
 that $\gra A$ is weak$\times$weak$^*$ closed.  Then $A$ is
 non-enlargeable if and only if $\gra (-A^*)\cap X\times
 X^*\subseteq\gra A$. In this situation, we have that $\langle
 x,x^*\rangle=0, \forall (x,x^*)\in\gra (-A^*)\cap X\times X^*$.
 \end{theorem}
 \begin{proof}
 ``$\Rightarrow$":
 By Corollary~\ref{CEl:1},
 \begin{align}F_A=\iota_{\gra A}+\langle\cdot,\cdot\rangle.\label{TheEF:1}
 \end{align}
  Let $(x,x^*)\in\gra(-A^*)\cap X\times X^*$. Then we have
 \begin{align}
 F_A(x,x^*)&=\sup_{(a,a^*)\in\gra A}\big\{\langle a^*,x\rangle
 +\langle a,x^*\rangle-\langle a,a^*\rangle\big\}\nonumber\\
 &=\sup_{(a,a^*)\in\gra A}\big\{-\langle a,a^*\rangle\big\}\nonumber\\
 &=0.\label{CRNon:1}
 \end{align}
 Then by \eqref{CRNon:1}, $(x,x^*)\in\gra A$ and $\langle x,x^*\rangle=0$.
 Hence $\gra (-A^*)\cap X\times X^*\subseteq\gra A$.

``$\Leftarrow$": By the assumption that $\gra A$ is weak$\times$weak$^*$ closed, we have
\begin{align}
\left[\gra(-A^*)\cap X\times X^*\right]^{\bot}\cap X^*\times X
=\left[\big(\gra A^{-1}\big)^{\bot}\cap X\times X^*\right]^{\bot}\cap X^*\times X=\gra A^{-1}.\label{WWSA:1}
\end{align}
By \cite[Lemma~2.1(2)]{SV}, we have
\begin{align}\langle z,z^*\rangle=0,\quad \forall(z,z^*)
\in \gra (-A^*)\cap X\times X^*.\label{TheEF:03}
 \end{align}
Hence $A^*|_X$ is skew.  Let $(x,x^*)\in X\times X^*$. Then by \eqref{TheEF:03},  we have
 \begin{align}
 F_A(x,x^*)&=\sup_{(a,a^*)\in\gra A}\big\{\langle x, a^*\rangle+
 \langle x^*, a\rangle-\langle a,a^*\rangle\big\}\nonumber\\
 &\geq\sup_{(a,a^*)\in\gra(-A^*)\cap X\times X^*}\big\{\langle x, a^*\rangle+
 \langle x^*, a\rangle-\langle a,a^*\rangle\big\}\nonumber\\
  &=\sup_{(a,a^*)\in\gra(-A^*)\cap X\times X^*}\big\{\langle x, a^*\rangle+
 \langle x^*, a\rangle\big\}\nonumber\\
 &=\iota_{\big(\gra(-A^*)\cap X\times X^*\big)^{\bot}\cap X^*\times X}(x^*,x)\nonumber\\
 &=\iota_{\gra A}(x,x^*)\quad\text{(by \eqref{WWSA:1})}.\label{SVCK:2}
 \end{align}
Hence by Fact~\ref{f:Fitz}
 \begin{align}
 F_A(x,x^*)=\langle x,x^*\rangle+\iota_{\gra A}(x,x^*).
 \label{SVCK:3}
 \end{align}
 Hence by Corollary~\ref{CEl:1}, $A$ is non-enlargeable.
 \end{proof}

The following corollary, which holds in a general Banach space, provides
a characterization of non-enlargeable operators under a
closedness assumption on the graph. A characterization
 of non-enlargeable linear operators  for reflexive spaces (in which
the closure assumption is hidden) was established by Svaiter in
\cite[Theorem~2.5]{SV}.

 \begin{corollary}\label{CEl:2}
Let  $A\colon X \To X^*$ be maximally monotone
and suppose that $\gra A$ is weak$\times$weak$^*$ closed.
Select  $(a,a^*)\in\gra A$ and set $\gra \widetilde{A}:=\gra A-\{(a,a^*)\}$.
Then
$A$ is non-enlargeable if and only if $\gra A$ is convex and
 $\gra (-\widetilde{A}^*)\cap X\times X^*\subseteq\gra \widetilde{A}$.
In particular, $\langle x,x^*\rangle=0, \forall (x,x^*)\in\gra \widetilde{A}^*\cap
X\times X^*$.
\end{corollary}
 \begin{proof} ``$\Rightarrow$":
 By the assumption that $A$ is non-enlargeable,
 so is $\widetilde{A}$.  By Fact~\ref{CElL:1}, $\gra A$ is convex
 and then $\gra A$ is affine by Fact~\ref{affine:L1}. Thus $\widetilde{A}$
 is a linear relation. Now we can apply Theorem~\ref{CElC:01} to $\widetilde{A}$.
 ``$\Leftarrow$": Apply Fact~\ref{affine:L1} and Theorem~\ref{CElC:01} directly.
 \end{proof}

\begin{remark}
We cannot remove the condition that ``$\gra A$ is convex" in
Corollary~\ref{CEl:2}. For example, let $X=\RR^n$ with the Euclidean
norm. Suppose that $f:=\|\cdot\|$.  Then $\partial f$ is maximally
monotone by Fact~\ref{SubMR}, and hence $\gra \partial f$ is
weak$\times$weak$^*$ closed. Now we show that
\begin{align}
\gra (\partial f)^*=\{(0,0)\}.\label{IExCo:01}
\end{align}
Note that
\begin{align}
\partial f(x)&=\begin{cases}B_X,\;&\text{if}\; x=0;\label{esee:3}\\
\{\frac{x}{\|x\|}\},\;&\text{otherwise}.\end{cases}
\end{align}
Let $(z,z^*)\in\gra (\partial f)^*$. By \eqref{esee:3}, we have $(0, B_X)\subseteq\gra \partial f$ and thus
\begin{align}
\langle -z, B_X\rangle=0.
\end{align}
Thus $z=0$. Hence
\begin{align}
\langle z^*, a \rangle=0,\quad \forall a\in \dom \partial f.\label{IExCo:1}
\end{align}
Since $\dom \partial f=X$, $z^*=0$ by \eqref{IExCo:1}. Hence
$(z,z^*)=(0,0)$ and thus \eqref{IExCo:01} holds. By \eqref{IExCo:01},
$\gra -(\partial f)^*\subseteq\gra \partial f$.  However, $\gra
\partial f$ is not convex.  Indeed, let
$e_k=(0,\ldots,0,1,0,\cdots,0):$ the $k$th entry is $1$ and the others
are $0$. Take \begin{align*}
  a=\frac{e_1-e_2}{\sqrt{2}}\quad\text{and}\quad
  b=\frac{e_2-e_3}{\sqrt{2}}.
\end{align*}
Then $(a, a)\in \gra \partial f$ and $(b, b)\in \gra \partial f$ by \eqref{esee:3}, but
\begin{align*}
\frac{1}{2}(a, a)+\frac{1}{2}(b, b)\notin\gra \partial f.
\end{align*}
Hence $\partial f$ is enlargeable by Fact~\ref{CElL:1}.
\end{remark}

In the case of a skew operator we can be more exacting:

\begin{corollary}\label{CElC:1}
Let  $A\colon X \To X^*$ be a maximally monotone and
skew operator and $\varepsilon\geq0$. Then
\begin{enumerate}
\item\label{EN:em01}
$\gra A_{\Elag}=\{(x,x^*)
\in\gra (-A^*)\cap X\times X^*\mid \langle x,x^*\rangle\geq-\varepsilon\}.$

\item \label{EN:em02}$A$ is non-enlargeable if and only if $\gra A=\gra (-A^*)\cap X\times X^*$.

\item \label{EN:em02b} $A$ is non-enlargeable if and only if $\dom A=\dom A^*\cap X$.
\item \label{EN:em03}Assume that $X$ is reflexive.  Then $F_{A^*}=\iota_{\gra A^*}+\scal{\cdot}{\cdot}$ and
hence $A^*$ is non-enlargeable.
\end{enumerate}
\end{corollary}
\begin{proof}
\ref{EN:em01}: By \cite[Lemma~3.1]{BBWY3}, we have
\begin{align}F_A=\iota_{\gra (-A^*)\cap X\times X^*}.\label{FPaSL:1}\end{align}
Hence $(x,x^*)\in\gra A_{\Elag}$ if and only if $F_A(x,x^*)\le \langle x,x^*\rangle + \ve$. This yields
$(x,x^*)\in \gra (-A^*)\cap X\times X^*$ and  $0\le \langle x,x^*\rangle + \ve$.

\ref{EN:em02}: From Fact~\ref{CElL:1} we have that $\dom F_A=\gra
A$. The claim now follows by combining the latter with \eqref{FPaSL:1}.

\ref{EN:em02b}: For  ``$\Rightarrow$": use ~\ref{EN:em02}.
``$\Leftarrow$": Since $A$ is skew, we have $\gra (-A^*)\cap X\times
X^*\supseteq \gra A$. Using this and \ref{EN:em02}, it suffices to
show that $\gra (-A^*)\cap X\times X^*\subseteq \gra A$.  Let
$(x,x^*)\in\gra (-A^*)\cap X\times X^*$. By the assumption, $x\in\dom
A$. Let $y^*\in Ax$. Note that $\langle x,-x^*\rangle=\langle
x,y^*\rangle=0$, where the first equality follows from the definition
of $A^*$ and the second one from the fact that $A$ is skew. In this
case we claim that $(x,x^*)$ is monotonically related to $\gra A$. Indeed, let
$(a,a^*)\in\gra A$.  Since $A$ is skew we have $\langle a, a^*\rangle= 0$. Thus
\[
\langle x-a,x^*-a^*\rangle=\langle x,x^*\rangle-\langle (x^*,x),
(a,a^*)\rangle+\langle a, a^*\rangle=0
\]
since $(x^*,x)\in (\gra A)^{\bot}$ and $\langle x,x^*\rangle=\langle a, a^*\rangle=0$. Hence
$(x,x^*)$ is monotonically related to $\gra A$. By maximality we conclude $(x,x^*)\in\gra A$.
Hence $\gra (-A^*)\cap X\times X^*\subseteq \gra A$.

\ref{EN:em03}: Now assume that $X$ is reflexive.  By
\cite[Theorem~2]{Brezis-Browder} (or see \cite{Yao, Si04}),
$A^*$ is maximally monotone.  Since $\gra A\subseteq \gra (-A^*)$ we
deduce that $\gra (-A^{**})=\gra (-A)\subseteq \gra A^*$. The latter
inclusion and Theorem~\ref{CElC:01} applied to the operator $A^*$
yields $A^*$ non-enlargeable.  The conclusion now follows by
applying Corollary~\ref{CEl:1} to $A^*$.
\end{proof}

\subsection{Limiting examples and remarks}

It is possible for a non-enlargeable maximally monotone operator to be
non-skew. This is the case for the operator $A^*$ in
Example~\ref{Exam:eL1}.

\begin{example}
Let $A\colon X \To X^*$ be a non-enlargeable maximally
 monotone operator. By Fact~\ref{CElL:1} and
Fact~\ref{affine:L1}, $\gra A$ is affine.  Let $f:X\rightarrow \RX$ be a proper
lower semicontinuous convex function with $\dom A\cap\inte\dom
\partial f\neq\varnothing$ such that $\dom A\cap\dom \partial f$ is
not an affine set.  By Fact~\ref{domain:L1}, $A+\partial f$ is
maximally monotone. Since $\gra (A+\partial f)$ is not affine,
$A+\partial f$ is enlargeable.\endproof
\end{example}

The operator in the following example was studied in detail in
\cite{BWY7}.

\begin{fact}\label{FE:1}
Suppose that  $X=\ell^2$,  and that
$A:\ell^2\rightrightarrows \ell^2$ is given by
\begin{align}Ax:=\frac{\bigg(\sum_{i< n}x_{i}-\sum_{i> n}x_{i}\bigg)_{n\in\NN}}{2}
=\bigg(\sum_{i< n}x_{i}+\tfrac{1}{2}x_n\bigg)_{n\in\NN},
\quad \forall x=(x_n)_{n\in\NN}\in\dom A,\label{EL:1}\end{align}
where $\dom A:=\Big\{ x:=(x_n)_{n\in\NN}\in \ell^{2}\mid \sum_{i\geq 1}x_{i}=0,
 \bigg(\sum_{i\leq n}x_{i}\bigg)_{n\in\NN}\in\ell^2\Big\}$ and $\sum_{i<1}x_{i}:=0$. Now \cite[Propositions~3.6]{BWY7} states that
\begin{align}
\label{PF:a2}
A^*x= \bigg(\thalb x_n + \sum_{i> n}x_{i}\bigg)_{n\in\NN},
\end{align}
where
\begin{equation*}
x=(x_n)_{n\in\NN}\in\dom A^*=\bigg\{ x=(x_n)_{n\in\NN}\in \ell^{2}\;\; \bigg|\;\;
 \bigg(\sum_{i> n}x_{i}\bigg)_{n\in\NN}\in \ell^{2}\bigg\}.
\end{equation*}
Then $A$ is an at most single-valued linear relation
such that the following hold (proofs of all claims are in brackets).
\begin{enumerate}
 \item
 $A$\label{NEC:1} is maximally monotone and skew (\cite[Propositions 3.5 and 3.2]{BWY7}).
\item\label{NEC:2} $A^*$ is maximally monotone but not skew
 (\cite[Theorem 3.9 and Proposition 3.6]{BWY7}).
\item \label{NEC:3}$\dom A$ is dense in $\ell^2$
 (\cite[Theorem 2.5]{PheSim}), and $\dom A\subsetneqq\dom A^*$ (\cite[Proposition 3.6]{BWY7}).
\item\label{NEC:5} $\langle A^*x, x\rangle=\tfrac{1}{2}s^2, \quad
  \forall x=(x_n)_{n\in\NN}\in\dom A^*\ \text{with}\quad
  s:=\sum_{i\geq1} x_i$ (\cite[Proposition 3.7]{BWY7}).
\end{enumerate}
\end{fact}

\begin{example}\label{Exam:eL1}
Suppose that $X$ and $A$ are as in Fact~\ref{FE:1}. Then $A$ is
enlargeable but $A^*$ is non-enlargeable and is not skew.
Moreover, \begin{align*}\gra A_{\Elag}=\big\{(x,x^*) \in\gra
  (-A^*)\mid\big|\sum_{i\geq1}
  x_i\big|\leq\sqrt{2\varepsilon},\ x=(x_n)_{n\in\NN}\big\},
\end{align*}
where $\varepsilon\geq0$.
\end{example}
\begin{proof}
By Corollary~\ref{CElC:1}(iii) and Fact~\ref{FE:1}\ref{NEC:3}, $A$
must be enlargeable.  For the second claim, note that $X=\ell^2$ is
reflexive, and hence by Fact~\ref{FE:1}\ref{NEC:1} and
Corollary~\ref{CElC:1}\ref{EN:em03}, for every skew operator we must
have $A^*$ non-enlargeable.  For the last statement, apply
Corollary~\ref{CElC:1}\ref{EN:em01} and Fact~\ref{FE:1}\ref{NEC:5}
directly to obtain $\gra A_{\Elag}$.
\end{proof}

\begin{example}\label{Exam:eN1}
Let $C$ be a nonempty closed convex subset of $X$ and $\varepsilon\geq 0$.
Then \begin{align*}
\gra (N_C)_{\Elag}=
\big\{(x,x^*)\in C\times X^*\mid \sigma_C (x^*)\leq \langle x,x^*\rangle+
\varepsilon\big\}.
\end{align*}
\end{example}
\begin{proof}
By Fact~\ref{f:referee04}, we have
\begin{align*}
(x,x^*)\in\gra \ (N_C)_{\Elag}
&\Leftrightarrow F_{N_C}(x,x^*)=\iota_C(x)+\sigma_C(x^*)
\leq \langle x,x^*\rangle+\varepsilon\\
&\Leftrightarrow x\in C,\ \sigma_C(x^*)
\leq \langle x,x^*\rangle+\varepsilon.
\end{align*}
\end{proof}

\begin{example}\label{Exam:eL02}
Let $f(x):=\|x\|,\;\forall x\in X$  and $\varepsilon\geq 0$. Then
$$\gra (\partial f)_{\Elag}= \big\{(x,x^*)\in X\times B_{X^*}
\mid  \|x\|\leq \langle x,x^*\rangle+\varepsilon\big\}.$$
In particular, $(\partial f)_{\Elag}(0)=B_{X^*}$.

\end{example}
\begin{proof}
Note that $f$ is sublinear, and hence by Fact \ref{f:sub05} and Remark
\ref{r:sub05} we can write
\begin{align*}(x,x^*)\in\gra (\partial f)_{\Elag}
&\Leftrightarrow  F_{\partial f}(x,x^*)=f(x)+f^*(x^*)
\leq \langle x,x^*\rangle+\varepsilon\quad\text{(by \eqref{Enl:Sub1})}\\
&\Leftrightarrow \|x\|+\iota_{B_{X^*}}(x^*)
\leq \langle x,x^*\rangle+\varepsilon\quad\text{(by \cite[Corollary~2.4.16]{Zalinescu})}\\
&\Leftrightarrow x^*\in B_{X^*},\   \|x\|\leq \langle x,x^*\rangle+\varepsilon.
\end{align*}
Hence $(\partial f)_{\Elag}(0)=B_{X^*}$.
\end{proof}

\begin{example}\label{Exam:eL2}
Let $p>1$ and  $f(x):=\tfrac{1}{p} \|x\|^p,\;\forall x\in X$. Then
$$(\partial f)_{\Elag}(0)= p^{\tfrac{1}{p}}(q\varepsilon)^{\tfrac{1}{q}}B_{X^*},$$
where $\tfrac{1}{p}+\tfrac{1}{q}=1$  and $\varepsilon\geq 0$.

\end{example}
\begin{proof}
We have
\begin{align*}x^*\in(\partial f)_{\Elag}(0)
&\Leftrightarrow\langle x^*-y^*,-y\rangle\geq-\varepsilon,\quad\forall y^*\in\partial f(y)\\
&\Leftrightarrow  \langle x^*,-y\rangle+\|y\|^{p}\geq-\varepsilon,
\quad\forall y\in X\\
&\Leftrightarrow \langle x^*,y\rangle-\|y\|^{p}
\leq\varepsilon,\quad\forall y\in X\\&\Leftrightarrow p\sup_{y\in X}
\Big[ \langle \tfrac{1}{p}x^*,y\rangle-\tfrac{1}{p}\|y\|^{p}\Big]\leq\varepsilon\\
&\Leftrightarrow
 p\cdot \tfrac{1}{q}\|\tfrac{1}{p} x^*\|^{q}\leq\varepsilon\\
 &\Leftrightarrow \|x^*\|^{q}\leq q\varepsilon p^{q-1}
 = q\varepsilon p^{\tfrac{q}{p}}\\
 &\Leftrightarrow x^*\in
   p^{\tfrac{1}{p}}(q\varepsilon)^{\tfrac{1}{q}} B_{X^*}.
\end{align*}
\end{proof}

\subsection{Applications of Fitzpatrick's last function}

For a monotone linear operator $A\colon X \rightarrow X^*$ it will be
very useful to define the following quadratic function (which is
actually a special case of \emph{Fitzpatrick's last function}
\cite{BorVan} for the linear relation $A$):
\begin{equation*}
q_A \colon x\mapsto \thalb \scal{x}{Ax}.
\end{equation*}
Then $q_A=q_{A_+}$.
We shall use the well known fact (see, e.g., \cite{PheSim}) that
\begin{equation} \label{e:gradq}
\nabla q_A = A_+,
\end{equation}
where the gradient operator
$\nabla$ is understood in the G\^ateaux sense.

The next result was first given in \cite[Proposition~2.2]{BWY2} for a
reflexive space. The proof is easily adapted to a general Banach
space.

\begin{fact}\label{better}
Let $A\colon X\to X^*$ be linear continuous,  symmetric and monotone.
Then
\begin{equation}\label{qa}
\big(\forall (x,x^*)\in X\times X^*\big)\quad
q_{A}^*(x^*+Ax)=q_{A}(x)+\scal{x}{x^*} +q_{A}^*(x^*)
\end{equation}
and $q_{A}^*\small\circ A=q_{A}$.
\end{fact}

The next result was first proven in \cite[Proposition~2.2(v)]{BBW} in
Hilbert space. We now extend it to a general Banach space.

\begin{proposition}\label{f1:Fitz}
Let   $A\colon X\to X^*$ be linear and monotone. Then
\begin{equation}
F_A(x,x^*)=2
q_{A_+}^*(\tfrac{1}{2}x^*+\tfrac{1}{2}A^*x)
=\tfrac{1}{2}q_{A_{+}}^*(x^*+A^*x),\quad \forall(x,x^*)\in X\times X,
\end{equation}
and $\ran A_+\subseteq\dom\partial q_{A_{+}}^* \subseteq\dom
q_{A_{+}}^*\subseteq\overline{\ran A_+}$. If $\ran A_{+}$ is closed,
then $\dom q_{A_{+}}^*=\dom\partial q_{A_{+}}^*=\ran A_{+}$.
\end{proposition}
\begin{proof}
By Fact~\ref{F:1}, $\dom A^*\cap X=X$, so for every $x,y\in X$ we have
$x,y\in \dom A^*\cap \dom A$. The latter fact and the definition of
$A^*$ yield $\scal{y}{A^*x}=\scal{x}{Ay}$. Hence for every $(x,x^*)\in
X\times X^*$,
\begin{align}
F_A(x,x^*) &= \sup_{y\in X} \scal{x}{Ay} +\scal{y}{x^*} -
\scal{y}{Ay}\notag\\
&= 2\sup_{y\in X} \scal{y}{\thalb x^* + \thalb A^*x} - q_{A_+}(y)\notag\\
&= 2q_{A_+}^*(\thalb x^*+\thalb A^*x)\notag\\
&= \thalb q_{A_+}^*(x^* + A^*x),
\end{align}
where we also used the fact that $q_A=q_{A_+}$ in the second
equality. The third equality follows from the definition of Fenchel
conjugate.  By \cite[Proposition~2.4.4(iv)]{Zalinescu},
\begin{align} \ran \partial q_{A_{+}}\subseteq \dom\partial q_{A_{+}}^*\label{suu:fix1}\end{align}
By \eqref{e:gradq},
$\ran \partial q_{A_{+}}=\ran A_+$. Then by \eqref{suu:fix1},
\begin{align}\ran A_+\subseteq\dom\partial q_{A_{+}}^*\subseteq\dom q_{A_{+}}^*\end{align}
Then by  the Br{\o}ndsted-Rockafellar Theorem
(see  \cite[Theorem~3.1.2]{Zalinescu}),
\begin{align*}\ran A_+\subseteq\dom\partial q_{A_{+}}^*
\subseteq\dom q_{A_{+}}^*\subseteq\overline{\ran A_+}.
\end{align*}
Hence, under the assumption that $\ran A_+$ is closed, we have
$\ran A_+=\dom\partial q_{A_{+}}^*=\dom q_{A_{+}}^*$.
\end{proof}

We can now apply the last proposition to obtain a formula for the enlargement of a single valued-operator.

\begin{proposition}[Enlargement of a monotone linear operator]\label{Bu:1}
Let $A:X\rightarrow X^*$ be a linear and  monotone operator, and $\varepsilon\geq0$. Then
\begin{align}A_{\Elag}(x)=\Big\{Ax+z^*\mid
 q^*_{A}(z^*)\leq 2\varepsilon\Big\},\quad \forall x\in X.\label{LSNE:1}\end{align}
Moreover, $A$ is non-enlargeable if and only if $A$ is skew.
\end{proposition}
\begin{proof} Fix $x\in X$, $z^*\in X^*$ and $x^*=Ax+z^*$.
 Then by Proposition~\ref{f1:Fitz} and Fact~\ref{better},
\begin{align*}&x^*\in A_{\Elag}(x)
\Leftrightarrow F_A(x,Ax+z^*)\leq\langle x,Ax+z^*\rangle+\varepsilon\\
&\Leftrightarrow \tfrac{1}{2}q^*_{A_+}(Ax+z^*+A^*x)
\leq\langle x,Ax+z^*\rangle+\varepsilon\\
&\Leftrightarrow \tfrac{1}{2}q^*_{A_+}\big(A_+ (2x)+z^*\big)
\leq\langle x,Ax+z^*\rangle+\varepsilon\\
&\Leftrightarrow \tfrac{1}{2}\left[q^*_{A_+}(z^*)+2\langle x,z^*\rangle+2\langle x, Ax\rangle\right]
\leq\langle x,Ax+z^*\rangle+\varepsilon\\
&\Leftrightarrow  q^*_{A}(z^*)\leq 2\varepsilon,
\end{align*}
where we also used in the last equivalence the fact that
$q_A=q_{A_+}$. Now we show the second statement. By Fact~\ref{F:1},
$\dom A^*\cap X=X$.  Then by Theorem~\ref{CElC:01} and
Corollary~\ref{CElC:1}\ref{EN:em02b}, we have $A$ is non-enlargeable
if and only if $A$ is skew.
\end{proof}

A  result similar  to Corollary~\ref{Tu:1} below
 was proved in \cite[Proposition~2.2]{BurIus:1}  in  reflexive space. Their proof still requires
   the constraint that $\ran (A+A^*)$ is closed.
\begin{corollary}\label{Tu:1}
Let $A:X\rightarrow X^*$ be a linear continuous and
  monotone operator such that $\ran (A+A^*)$ is closed. Then
$$A_{\Elag}(x)=\Big\{Ax+(A+A^*)z\mid q_{A}(z)\leq \tfrac{1}{2}\varepsilon\Big\},\quad \forall x\in X.$$
\end{corollary}
\begin{proof}
Proposition~\ref{Bu:1} yields
\begin{equation}\label{Tu:eq1}
x^*\in A_{\Elag}(x)\Leftrightarrow x^*=Ax+z^*,\; q^*_{A}(z^*)\leq 2\varepsilon.
\end{equation}
In particular, $z^*\in \dom q^*_{A}$. Since $\ran (A_{+})$ is
closed, Proposition~\ref{f1:Fitz} yields
\[
\ran (A_{+}) = \ran (A+A^*)=\dom q^*_{A_+}=\dom q^*_{A}.
\]
The above expression and the fact that $z^*\in \dom q^*_{A}$ implies
that there exists $z\in X$ such that $z^*=(A+A^*)z$. Note also that (by Fact~\ref{better})
\[
 q^*_{A}(z^*)=q^*_{A_+}(z^*)=q^*_{A_+}(A_+(2z))=q_{A_+}(2z)=4q_A(z),
\]
where we used Fact~\ref{better} in the last equality. Using this in \eqref{Tu:eq1} gives
\begin{align*}
x^*\in A_{\Elag}(x)&\Leftrightarrow
x^*=Ax+(A+A^*)z,\;4q_{A}(z)\leq 2\varepsilon\\
&\Leftrightarrow x^*=Ax+(A+A^*)z,\;q_{A}(z)\leq \tfrac{1}{2}\varepsilon,\end{align*}
establishing the claim.
\end{proof}

We conclude the section with two examples.

\begin{example}[Rotation]\label{Exam:eL3}
Assume  that $X$ is the Euclidean plane $\RR^2$,
let $\theta \in\left[0,\tfrac{\pi}{2}\right]$, and set
\begin{align}
A:=\begin{pmatrix}
\cos\theta & -\sin\theta \\
\sin \theta  & \cos\theta
\end{pmatrix}.
\end{align}

Then for every $(\varepsilon,x)\in\RR_+\times\RR^2$,
\begin{align}A_{\Elag}(x)=\Big\{Ax+v\mid v\in
2\sqrt{(\cos{\theta})\varepsilon\,}B_X\Big\}.\label{ComSE:1}
\end{align}
\end{example}

\begin{proof}
We consider two cases.

\emph{Case 1}: $\theta=\tfrac{\pi}{2}$.

Then $A$ is skew operator.
By Corollary~\ref{CElC:1}, $A_{\Elag}=A$
and hence \eqref{ComSE:1} holds.

\emph{Case 2}:  $\theta \in\left[0,\tfrac{\pi}{2}\right[$.

Let $x\in\RR^2$.
Note that $\tfrac{A+A^*}{2}= (\cos\theta)\Id$, $q_{A}
 = \tfrac{\cos\theta}{2}\|\cdot\|^2$. Then by Corollary~\ref{Tu:1},
$$A_{\Elag}(x)=\Big\{Ax+2(\cos{\theta}) z\mid q_{A}(z)
=\tfrac{\cos\theta}{2}\|z\|^2\leq \tfrac{1}{2}\varepsilon\Big\}.$$
Thus, \begin{align*}A_{\Elag}(x)&=\Big\{Ax+v\mid
 \|v\|\leq 2\sqrt{(\cos\theta)\varepsilon\,}\Big\}
=\Big\{Ax+v\mid v\in 2\sqrt{(\cos\theta)\varepsilon\,}B_X\Big\}.
\end{align*}

\end{proof}

\begin{example}[Identity]\label{Exam:eL4}
Assume  that $X$ is a Hilbert space, and $A:=\Id$.
Let $\varepsilon\geq0$. Then
\begin{align*}\gra A_{\Elag}
=\Big\{(x,x^*)\in X\times X\mid x^*\in x+ 2\sqrt{\varepsilon}B_X\Big\}.
\end{align*}
\end{example}
\begin{proof}
By \cite[Example~3.10]{BM}, we have
\begin{align*}
(x,x^*)\in\gra A_{\Elag}
&\Leftrightarrow \tfrac{1}{4}\|x+x^*\|^2\leq\langle x,x^*\rangle+\varepsilon\\
&\Leftrightarrow \tfrac{1}{4}\|x-x^*\|^2\leq\varepsilon\\
&\Leftrightarrow\|x-x^*\|\leq 2\sqrt{\varepsilon}\\
&\Leftrightarrow x^*\in x+ 2\sqrt{\varepsilon}B_X.
\end{align*}
\end{proof}

\section{Sums of operators}
\label{secsumo}

The conclusion of the lemma below has been established for reflexive
Banach spaces in \cite[Lemma~5.8]{BWY3}.  Our proof for a general
Banach space assumes the operators to be of type (FPV) and follows
closely that of \cite[Lemma~5.8]{BWY3}.

\begin{lemma}\label{Co:1}
Let $A, B\colon X\To X^*$ be  maximally monotone of type (FPV), and suppose that
$\bigcup_{\lambda>0}\lambda\left[\dom A-\dom B\right]$ is a closed subspace of $X$.
Then we have
\begin{align*}\bigcup_{\lambda>0}\lambda\left[\dom A-\dom B\right]=
\bigcup_{\lambda>0}\lambda\left[P_{X}\dom F_A-P_{X}\dom F_B\right].\end{align*}
\end{lemma}
\begin{proof}
By Fact~\ref{f:Fitz} and Fact~\ref{f:referee02c}, we have
\begin{align*}
&\bigcup_{\lambda>0} \lambda\left[\dom A-\dom B\right]\subseteq
\bigcup_{\lambda>0} \lambda\left[P_{X}\dom F_A-P_{X}\dom F_B\right]
\subseteq\bigcup_{\lambda>0} \lambda\left[\overline{\dom A}-\overline{\dom B}\right]\\
&\subseteq\bigcup_{\lambda>0} \lambda\left[\overline{\dom A-\dom B}\right]\subseteq
\overline{\bigcup_{\lambda>0} \lambda\left[\dom A-\dom B\right]}\\
&=\bigcup_{\lambda>0} \lambda\left[\dom A-\dom B\right]\quad \text{(by the  assumption)}.
\end{align*}
\end{proof}

\begin{corollary}\label{Co:01}
Let $A, B\colon X\To X^*$ be  maximally monotone linear relations, and suppose that
$\dom A-\dom B$ is a closed subspace.
Then
\begin{align*}\left[\dom A-\dom B\right]=\bigcup_{\lambda>0} \lambda\left[P_{X}\dom F_A-P_{X}\dom F_B\right].
\end{align*}
\end{corollary}
\begin{proof} Directly apply  Fact~\ref{f:referee01} and Lemma~\ref{Co:1}.
\end{proof}

\begin{corollary}\label{Co:01sd}
Let $A\colon X\To X^*$ be a  maximally monotone linear relation
 and let $C\subseteq X$ be a nonempty and closed convex set. Assume that
$\bigcup_{\lambda>0} \lambda\left[\dom A-C\right]$ is a closed subspace.
Then
\begin{align*}\bigcup_{\lambda>0} \lambda\left[P_{X}\dom F_A-P_{X}\dom F_{N_C}\right]
&=\bigcup_{\lambda>0} \lambda\left[\dom A-C\right].
\end{align*}
\end{corollary}
\begin{proof} Let $B=N_C$. Then apply directly
 Fact~\ref{f:referee01}, Fact~\ref{f:referee0d} and Lemma~\ref{Co:1}.
\end{proof}

Theorem~\ref{FS6} below was  proved  in \cite[Theorem~5.10]{BWY3} for  a reflexive space.
We  extend it  to a general Banach space.

\begin{theorem}[Fitzpatrick function of the sum]\label{FS6}
Let $A,B\colon X\To X^*$ be maximally monotone linear relations,
and suppose that $\dom A-\dom B$ is closed.
Then $$F_{A+B}= F_A\Box_2F_B,$$
and the partial infimal convolution is exact everywhere.
\end{theorem}
\begin{proof}Let $(z,z^*)\in X\times X^*$. By Fact~\ref{f:referee03}, it suffices to show that there exists $v^*\in X^*$ such that
\begin{equation} \label{elrl:ourgoal}
F_{A+B}(z,z^*)\geq F_A (z,z^*-v^*)+ F_{B}(z,v^*).
\end{equation}
If $(z,z^*)\notin \dom F_{A+B}$, clearly, \eqref{elrl:ourgoal} holds.

Now assume that $(z,z^*)\in \dom F_{A+B}$. Then
\begin{align}
&F_{A+B}(z,z^*)\nonumber\\
&=\sup_{\{x,x^*,y^*\}}\big[\langle x,z^*\rangle+\langle z,x^*\rangle-\langle x,x^*\rangle
+\langle z-x, y^*\rangle -\iota_{\gra A}(x,x^*)-\iota_{\gra
B}(x,y^*)\big].\label{lrsee:1}
\end{align}
Let $Y=X^*$ and define $F,K:X\times X^*\times Y\rightarrow\RX$ respectively by
\begin{align*}
F: &(x,x^*,y^*)\in X\times X^*\times Y\rightarrow\langle x,x^*\rangle+\iota_{\gra A}(x,x^*)\\
K:&(x,x^*, y^*)\in X\times X^*\times Y\rightarrow\langle x,y^*\rangle+\iota_{\gra B}(x,y^*)\\
\end{align*}
Then by \eqref{lrsee:1},
\begin{align}F_{A+B}(z,z^*)=(F+K)^*(z^*,z,z)\label{lrsee:2}\end{align}
By Fact~\ref{f:referee} and the assumptions, $F$ and
 $K$ are proper lower semicontinuous and convex. The definitions of $F$ and $K$ yield
\begin{align*}
\dom F-\dom K=\left[\dom A-\dom B\right]\times X^*\times Y,\ \text{ which is a closed subspace}.
\end{align*}
Thus by
Fact~\ref{AttBre:1} and \eqref{lrsee:2}, there exists $(z^*_0
,z^{**}_0,z^{**}_1)\in X^*\times X^{**}\times Y^*$ such that
 \begin{align*}F_{A+B}(z,z^*)&= F^*(z^*-z^*_0,z-z^{**}_0,z-z^{**}_1)+K^*(z^*_0
,z^{**}_0,z^{**}_1)\\
&= F^*(z^*-z^*_0,z,0)+K^*(z^*_0,0,z)\quad\text{(by $(z,z^*)\in\dom F_{A+B}$)}\\
&=F_A(z,z^*-z^*_0)+F_B(z,z^*_0).
\end{align*}
Thus \eqref{elrl:ourgoal} holds by taking $v^*=z_0^*$ and hence $F_{A+B}=F_A\Box_2 F_B$.
\end{proof}

The next result  was  first obtained by  Voisei in \cite{Voisei06} while
 Simons gave a different proof in \cite[Theorem~46.3]{Si2}.
We  are now in position to provide a third approach.

\begin{theorem}\label{lisum:1}
Let $A,B\colon X\To X^*$ be maximally monotone linear relations,
and suppose that $\dom A-\dom B$ is closed.
Then $A+B$ is maximally monotone.
\end{theorem}
\begin{proof}
By Fact~\ref{f:Fitz}, we have that $F_A\geq \langle\cdot,\cdot\rangle$
and $F_B\geq \langle\cdot,\cdot\rangle$.  Using now Theorem~\ref{FS6}
and \eqref{infconv} implies that
$F_{A+B}\geq\langle\cdot,\cdot\rangle$. Combining the last inequality
with Corollary~\ref{Co:01} and Fact~\ref{f:referee1}, we conclude that
$A+B$ is maximally monotone.
\end{proof}

\begin{theorem}\label{lisum:e1}
Let $A,B\colon X\To X^*$ be maximally monotone linear relations,
and suppose that $\dom A-\dom B$ is closed.  Assume that $A$ and $B$ are non-enlargeable.
Then  $$F_{A+B}=\iota_{\gra(A+B)}+\langle\cdot,\cdot\rangle$$ and hence $A+B$ is non-enlargeable.
\end{theorem}
\begin{proof}
By Corollary~\ref{CEl:1}, we have
\begin{align}
F_A=\iota_{\gra A}+\langle\cdot,\cdot\rangle\quad\text{and}\quad
F_B=\iota_{\gra B}+\langle\cdot,\cdot\rangle.\label{SumEL:1}
\end{align}
Let $(x,x^*)\in X\times X^*$.
Then by \eqref{SumEL:1} and Theorem~\ref{FS6}, we have
\begin{align*}
F_{A+B} (x,x^*)&=\min_{y^*\in X^*}\big\{\iota_{\gra A}(x,x^*-y^*)+\langle
x^*-y^*,x\rangle+\iota_{\gra B}(x,y^*)+\langle y^*,x\rangle\big\}\\
&=\iota_{\gra (A+B)}(x,x^*)+\langle x^*,x\rangle.
\end{align*}
By Theorem~\ref{lisum:1} we have that $A+B$ is maximally monotone. Now
we can apply Corollary~\ref{CEl:1} to $A+B$ to conclude that $A+B$ is
non-enlargeable.
\end{proof}

The proof  of Theorem~\ref{tf:main}  in part follows that of \cite[Theorem~3.1]{BWY4}.

\begin{theorem}\label{tf:main}
Let $A:X\To X^*$ be a maximally monotone linear relation.
Suppose $C$ is a nonempty closed convex subset of $X$,
and  that $\dom A \cap \inte C\neq \varnothing$.
Then $F_{A+N_C}= F_A\Box_2F_{N_C}$,
and the partial infimal convolution is exact everywhere.
\end{theorem}

\begin{proof}
Let $(z,z^*)\in X\times X^*$.
By Fact~\ref{f:referee03}, it suffices to show that there exists $v^*\in X^*$ such that
\begin{equation} \label{e:ourgoal}
F_{A+N_C}(z,z^*)\geq F_A (z,v^*)+ F_{N_C}(z,z^*-v^*).
\end{equation}
If $(z,z^*)\notin \dom F_{A+N_C}$, clearly, \eqref{e:ourgoal} holds.

Now assume that \begin{align}(z,z^*)\in \dom F_{A+N_C}.\label{EacInm}
\end{align}
By Fact~\ref{domain:L1} and Fact~\ref{f:referee02c},
 \begin{align*}P_X\left[\dom F_{A+N_C}\right]\subseteq
\overline{\left[\dom (A+N_C)\right]}\subseteq C.\end{align*}
Thus, by \eqref{EacInm}, we have
\begin{align}z\in C.\label{EacInm:1}
\end{align}
Set
\begin{equation}\label{e:defofg}
g\colon X\times X^* \to \RX\colon
(x,x^*)\mapsto \scal{x}{x^*} + \iota_{\gra A}(x,x^*).
\end{equation}
By Fact~\ref{f:referee}, $g$ is convex.
Hence,
\begin{equation} \label{e:defofh}
h = g + \iota_{C\times X^*}
\end{equation}
is convex as well.
Let
\begin{equation} \label{e:defofc0}
c_0 \in \dom A \cap \inte C,
\end{equation}
and let $c_0^*\in Ac_0$.
Then $(c_0,c_0^*)\in \gra A \cap (\inte C \times X^*) = \dom g \cap
\intdom\iota_{C\times X^*}$.  Let us compute $F_{A+N_C}(z,z^*)$. As in \eqref{lrsee:1} we can write
\begin{align}
&F_{A+N_C}(z,z^*)\nonumber\\
&=\sup_{(x,x^*,c^*)}\big[\langle x,z^*\rangle+\langle z,x^*\rangle-\langle x,x^*\rangle
+\langle z-x, c^*\rangle -\iota_{\gra A}(x,x^*)-\iota_{\gra N_C}(x,c^*)\big]\nonumber\\
&\geq \sup_{(x,x^*)}\big[\langle x,z^*\rangle+\langle z,x^*\rangle-\langle x,x^*\rangle
 -\iota_{\gra A}(x,x^*)-\iota_{C\times X^*}(x,x^*)\big]\nonumber\\
&=\sup_{(x,x^*)}\left[\langle x,z^*\rangle+\langle z,x^*\rangle-h(x,x^*)\right]\nonumber\\
&=h^*(z^*,z),\nonumber
\end{align}
where we took $c^*=0$ in the inequality. By Fact~\ref{pheps:1},
 $\iota_{C\times X^*}$ is continuous at
$(c_0,c_0^*)\in \intdom\iota_{C\times X^*}$. Since $(c_0,c_0^*)\in \dom g \cap
\intdom\iota_{C\times X^*}$ we can use Fact~\ref{f:F4} to conclude the
existence of $(y^*,y^{**})\in X^{*}\times X^{**}$ such that
\begin{align}
h^*(z^*,z)&=g^*(y^*,y^{**}) + \iota_{C\times X^*}^*(z^*-y^*,z-y^{**})\nonumber\\
&=g^*(y^*,y^{**}) + \iota_{C}^*(z^*-y^*) + \iota_{\{0\}}(z-y^{**}).\label{EacInm:2}
\end{align}
Then by \eqref{EacInm} and \eqref{EacInm:2} we must have
$z=y^{**}$.  Thus by \eqref{EacInm:2} and the definition of $g$ we have
\begin{align*}&F_{A+N_C}(z,z^*)\geq g^*(y^*,z) + \iota_{C}^*(z^*-y^*)
=F_A(z,y^*) + \iota_{C}^*(z^*-y^*)\\
&=F_A(z,y^*) + \iota_{C}^*(z^*-y^*)+\iota_C(z)\quad\text{(by \eqref{EacInm:1})}\\
&=F_A(z,y^*) + F_{N_C}(z,z^*-y^*)\quad\text{(by Fact~\ref{f:referee04})}.
\end{align*}
Hence \eqref{e:ourgoal} holds by taking $v^*=y^*$ and thus $F_{A+N_C}= F_A\Box_2F_{N_C}$.
\end{proof}

We decode the prior result as follows:

\begin{corollary}[Normal cone]\label{FLiN:1}
Let $A:X\To X^*$ be a maximally monotone linear relation.
Suppose $C$ is a nonempty closed convex subset of $X$,
and that $\dom A \cap \inte C\neq \varnothing$.
Then $A+N_C$ is maximally monotone.
\end{corollary}

\begin{proof}
By Fact~\ref{f:Fitz}, we have that $F_A\geq \langle\cdot,\cdot\rangle$
and $F_{N_C}\geq \langle\cdot,\cdot\rangle$.  Using now
Theorem~\ref{tf:main} and \eqref{infconv} implies that
$F_{A+N_C}\geq\langle\cdot,\cdot\rangle$. Combining the last
inequality with Corollary~\ref{Co:01} and Fact~\ref{f:referee1}, we
conclude that $A+N_C$ is maximally monotone.
\end{proof}

To conclude we revisit a quite subtle example. All statements in the
fact below have been proved in \cite[Example~4.1 and
  Theorem~3.6(vii)]{BBWY3}.

\begin{fact}\label{FCPEX:1}

 Consider  $X: = c_0$, with norm $\|\cdot\|_{\infty}$ so that
  $X^* = \ell^1$ with norm $\|\cdot\|_{1}$,
 and  $X^{**}=\ell^{\infty}$  with  second dual norm
$\|\cdot\|_{*}$. Fix
$\alpha:=(\alpha_n)_{n\in\NN}\in\ell^{\infty}$ with $\limsup
\alpha_n\neq0$, and define
$A_{\alpha}:\ell^1\rightarrow\ell^{\infty}$  by
\begin{align}\label{def:Aa}
(A_{\alpha}x^*)_n:=\alpha^2_nx^*_n+2\sum_{i>n}\alpha_n \alpha_ix^*_i,
\quad \forall x^*=(x^*_n)_{n\in\NN}\in\ell^1.\end{align}
\allowdisplaybreaks  Finally, let $T_{\alpha}:c_{0}\rightrightarrows X^*$  be defined by
\begin{align}\gra T_{\alpha}&:
=\big\{(-A_{\alpha} x^*,x^*)\mid x^*\in X^*,
 \langle \alpha, x^*\rangle=0\big\}\nonumber\\
&=\Big\{\big((-\sum_{i>n}
\alpha_n \alpha_ix^*_i+\sum_{i<n}\alpha_n \alpha_ix^*_i)_n, x^*\big)
\mid x^*\in X^*, \langle \alpha, x^*\rangle=0\Big\}.\label{PBABA:Ea1}
\end{align}
Then
\begin{enumerate}
\item\label{BCCE:A01} $\langle A_{\alpha}x^*,x^*\rangle=\langle \alpha , x^*\rangle^2,
\quad \forall x^*=(x^*_n)_{n\in\NN}\in\ell^1$
and so \eqref{PBABA:Ea1} is well defined.
\item \label{BCCE:SA01} $A_{\alpha}$ is a maximally monotone
 operator on $\ell^1$.

\item\label{BCCE:A1} $T_{\alpha}$
is a maximally monotone and skew operator on $c_0$.
\item\label{BCCE:A1F} $F_{T_{\alpha}}=\iota_C$, where
$
C:=\{(-A_{\alpha}x^*,x^*)\mid x^*\in X^*\}.
$
\end{enumerate}
 \end{fact}

This set of affairs allows us to show the following:

 \begin{example}
 Let $X=c_0$, $A_{\alpha}$, $C$, and $T_{\alpha}$ be defined as in Fact~\ref{FCPEX:1}.
 Then  $T_{\alpha}:c_{0}\rightrightarrows \ell^1$ is a maximally monotone enlargeable skew linear relation.
  Indeed
 \begin{align*}
 \gra (T_{\alpha}+N_{B_X})_{\Elag}=
 \Big\{(-A_{\alpha}x^*,z^*)\in B_X\times X^*\mid x^*\in X,
  \|z^*-x^*\|_1\leq\langle -A_{\alpha}x^*,z^*\rangle+\varepsilon\Big\}.
 \end{align*}
 \end{example}
\allowdisplaybreaks
\begin{proof}
From \eqref{PBABA:Ea1}, we have that $\gra T_{\alpha}\subsetneqq C$
therefore Fact~\ref{FCPEX:1}\ref{BCCE:A1F} yields $F_{T_{\alpha}}\neq
\iota_{\gra T_{\alpha}} +\langle\cdot,\cdot\rangle$. Using now
Fact~\ref{FCPEX:1}\ref{BCCE:A1} and Corollary~\ref{CEl:1}, we conclude
that $T_{\alpha}$ is enlargeable.

Now we determine $\gra (T_{\alpha}+N_{B_X})_{\Elag}$.
By Fact~\ref{FCPEX:1}\ref{BCCE:A1}, Theorem~\ref{tf:main} and \eqref{Enl:1}, we have
\begin{align}
&(z,z^*)\in \gra (T_{\alpha}+N_{B_X})_{\Elag}\nonumber\\
&\Leftrightarrow F_{T_{\alpha}}\Box_2F_{N_{B_X}}(z,z^*)\leq\langle z,z^*\rangle+\varepsilon
 \nonumber\\
&\Leftrightarrow F_{T_{\alpha}}(z,x^*)+\iota_{B_X}(z)+\iota^*_{B_X}(z^*-x^*)
\leq\langle z,z^*\rangle+\varepsilon, \;\exists x^*\in X^*  \quad\text{(by Fact~\ref{f:referee04})}\nonumber\\
&\Leftrightarrow z\in B_X,\, \iota_C(z,x^*)+\|z^*-x^*\|_1
\leq\langle z,z^*\rangle+\varepsilon,   \;\exists x^*\in X^*\;\text{(by Fact~\ref{FCPEX:1}\ref{BCCE:A1F})}\nonumber\\
&\Leftrightarrow z=-A_{\alpha}x^*\in B_X,\ \|z^*-x^*\|_1
\leq\langle z,z^*\rangle+\varepsilon,  \;\exists x^*\in X^*\nonumber\\
&\Leftrightarrow z=-A_{\alpha}x^*\in B_X,\;
  \|z^*-x^*\|_1\leq\langle -A_{\alpha}x^*,z^*\rangle+\varepsilon , \;\exists x^*\in X^*.\nonumber
\end{align}
This is the desired result.
\end{proof}

 \vfill

\paragraph{Acknowledgments.} The authors thank Dr.\ Heinz Bauschke and Dr.\ Xianfu Wang for their
valuable discussions and comments. Jonathan  Borwein was partially
supported by the Australian Research  Council.
The third author thanks CARMA at the University of Newcastle and  the School of Mathematics and Statistics of University of South Australia for the support of  his visit to Australia, which started this research.
\end{document}